\documentclass[fleqn]{article}
\usepackage{amsmath}
\usepackage{amsthm}
\usepackage[pdftex]{graphicx}
\usepackage{indentfirst}
\usepackage{color}
\usepackage{enumerate}
\usepackage[english]{babel}

\usepackage{algpseudocode}

\usepackage[utf8]{inputenc}
\usepackage{amsfonts}
\usepackage{amssymb}
\usepackage{fancyhdr}
\usepackage{setspace}
\usepackage{float}
\usepackage{theoremref}
\usepackage{comment}
\usepackage{xcolor}
\usepackage[numbers]{natbib}
\usepackage[toc,page]{appendix}

\usepackage{soul}

\large\normalsize
\numberwithin{equation}{section}
\usepackage[toc,page]{appendix}

\theoremstyle{plain}
\newtheorem{theorem}{Theorem}[section]
\newtheorem{lemma}[theorem]{Lemma}
\newtheorem{cor}[theorem]{Corollary}

\theoremstyle{definition}

\begin{document}
\newcommand{\T}{\mathbb{T}}
\newcommand{\R}{\mathbb{R}}
\newcommand{\Q}{\mathbb{Q}}
\newcommand{\N}{\mathbb{N}}
\newcommand{\Z}{\mathbb{Z}}
\newcommand{\tx}[1]{\quad\mbox{#1}\quad}

\title{On Reliability of Stochastic Networks}
\author{Ryan T. White\\
{\color{blue}rwhite2009@fit.edu}
\and
Jewgeni H. Dshalalow\\
{\color{blue}eugene@fit.edu}}
\date{}
\maketitle

\vspace{-1cm}
\begin{center}
Department of Mathematical Sciences\\
College of Science\\
Florida Institute of Technology\\
Melbourne, Florida 32901, USA
\end{center}

\begin{abstract}
In the recent times we hear increasingly often about cyber attacks on various commercial and strategic sites that manage to escape any defense. In this article we model such attacks on networks via stochastic processes and predict the time of a total or partial failure of a network including the magnitude of losses (such as the number of compromised nodes, lost weights, and a loss of other associated components relative to some fixed thresholds). To make such modeling more realistic we also assume that the information about the attacks is delayed as per random observations. We arrive at analytically and numerically tractable results demonstrated by examples and comparative simulation.

\end{abstract}

\section{INTRODUCTION}\label{s1}
In this paper, we consider a model of a large-scale stochastic network under a series of cyber attacks wherein successive random subgraphs are compromised (destroyed or otherwise prevented from normal operation) upon random time increments. With each node is an associated random weight representing the value of the node. Furthermore, we do not learn in real-time the extent of the damage: rather, the status of the cumulative damage is ascertained only upon an independent delayed renewal observation process.

Random graphs are very common in modeling various types of networks \lbrack 6-7, 23, 27, 29-30, 32, 36\rbrack, but few had cyber-crime as a focus \lbrack 28, 31\rbrack. The classical random graphs of Erd\"os and R\a`enyi, $G(n,M)$, consist of $n$ vertices with $M$ edges chosen uniformly at random from the set of all possible adjacencies among the $n\,$vertices \lbrack 17-18\rbrack, but we consider rather pre-existing random graphs with randomly weighted nodes. As such, we model each individual attack as removing some random number of nodes along with their total weight.

Models of cyber attacks sometimes consider a viral process spreading from one node to another according to a branching process. While this can lead to valuable insights in some domains, we aim to capture the incapacitation of institutional network assets under existing defensive strategies (firewalls, quarantining affected machines, and so forth), which result in uncertain graphs on which viral attacks may spread and the loss of operational capacity of potentially non-infected nodes, rendering the estimates of infected nodes an incomplete picture of losses.

Also, there may be multiple sources of viral and non-viral attacks (e.g. distributed denial of service). Further, batches of nodes are commonly lost because (1) some viral attacks, once beyond a firewall, can spread to an entire subnet very quickly \lbrack 8\rbrack, (2) viral detection often prompts administrators to quarantine subnets, (3) attacks may knock out a hub necessary for proper operation of adjacent nodes (e.g. a router) \lbrack 4\rbrack, and (4) many communication networks can be characterized by scale-free networks, yielding clusters of highly interconnected groups of nodes \lbrack 22, 28, 31\rbrack\ associated with particular subnets, increasing the threat of practically immediate internal contamination.

The primary target of our analysis will be the process in the vicinity of the first time that either a cumulative node loss component reaches a threshold $M$ or a cumulative weight loss component reaches a threshold $V$. Formally, this will be an exit time of a multivariate marked point process with mutually dependent components from the open rectangle $[0,M)\times[0,V)$. As such, we draw upon the extensive literature on fluctuation theory \lbrack 19-20, 24-26, 33-34\rbrack\ and properties of exit times of stochastic processes \lbrack 1-3, 10-16, 21-22, 35\rbrack.

The crossing of these critical thresholds correspond to points at which network activity undergoes some important change in operation, whether it corresponds to the detection of malicious attacks as opposed to benign losses (e.g. temporary maintenance or ordinary hardware failures), the point at which the situation dictates a change in defensive policy, or the destruction of the network. Thus, we want to predict when such crossings will happen, the values of the components of the process at the crossing, and the like.

Now, since the present paper focuses on application of previously obtained general results \lbrack 10,16\rbrack\ for random walks to the prediction of a failure of networks, we show in Section 3 that if the attacks form a marked Poisson process, special cases under delayed observations are tractable and they agree with results obtained by direct simulation.
\section{EXIT TIME MODEL FOR NETWORKS}\label{s2}
Let $\left(\Omega,\mathcal{F}\left(\Omega\right),P\right)$ be a probability space and let
\begin{equation}
\left(\mathcal{N}\otimes\mathcal{W}\otimes\mathcal{P}\right)=\sum_{k\geq 0}\left(n_k,w_k,\boldsymbol{p}_k\right)\varepsilon_{t_k}
\end{equation}
\noindent(where $\varepsilon_c\,$is a Dirac point measure) be a marked atomic random measure describing the evolution of damage to a network where attacks arrive upon each $t_k$, at which $n_k$ nodes are incapacitated and their associated weights, $w_k=w_{k1}+...+w_{kn_k}$ ($w_{kj}$ represents the nonnegative real weight of the $j$th node destroyed in the $k$th attack).

In addition, a real random $n$-vector representing the change in passive components for each node lost, yielding the last component of the mark: $\boldsymbol{p}_k=\boldsymbol{p}_{k1}+...+\boldsymbol{p}_{kn_k}$, where $\boldsymbol{p}_{kj}$ represents the change in the $n$ passive components due to loss of the $j$th node of the $k$th attack). Altogether, we have
\begin{equation}
\left(n_k,w_k,\boldsymbol{p}_k\right):\Omega\rightarrow\mathbb{N}\times\mathbb{R}_0\times\mathbb{R}^n
\end{equation}

We assume the increments $\left(n_k,w_k,\boldsymbol{p}_k\right)$ are jointly independent and identically distributed (iid) for $k\in\mathbb{N}$ and each is independent of $\left(n_0,w_0,\boldsymbol{p}_0\right)$, though components are mutually dependent.

Furthermore, we assume the common probability-generating function (PGF) of each $n_k$ is $g\left(z\right)$, the common Laplace-Stieltjes transform (LST) of each $w_{kj}$ is $l\left(v\right)$, and the common $n$-variate moment-generating function of each $p_{kj}$ is $m\left(\boldsymbol{\alpha}\right)$. In addition, we assume each $w_k$ and $\boldsymbol{p}_k$ are conditionally independent given $n_k$.

We define the joint transform of the increments as
\begin{equation}
\gamma\left(z,v,\boldsymbol{\alpha}\right)=E\left[z^{n_k}e^{-vw_k}e^{\boldsymbol{\alpha\cdot p}_k}\right]
\end{equation}
where $k\in\mathbb{N},\,|z|\leq 1,v\in\mathbb{R},\,\alpha\in\mathbb{R}^m$.

Using the transforms above with iterated expectation and the conditional independence of $w_k$ and $\boldsymbol{p}_k$ given $n_k$, we can simplify this as
\begin{equation}
\gamma\left(z,v,\boldsymbol{\alpha}\right)=g\left[zl\left(v\right)m\left(\boldsymbol{\alpha}\right)\right]
\end{equation}

Consider the continuous time parameter process associated with the random measure $\mathcal{N}\otimes\mathcal{W}\otimes\mathcal{P}$ introduced in (2.1),
\begin{equation}
\left(N\left(t\right),W\left(t\right),\mathfrak{P}\left(t\right)\right)=\left(\mathcal{N}\otimes\mathcal{W}\otimes\mathcal{P}\right)\left[0,t\right]
\end{equation}

However, we will suppose the process is observed only upon the following delayed renewal process (rather than in real-time)
\begin{equation}
\mathcal{T}=\sum_{k\geq 0}\varepsilon_{\tau_k}
\end{equation}
\noindent where $\Delta_k=\tau_k-\tau_{k-1}$ for $k\in\mathbb{N}_0$ and $\tau_{-1}=0,\,$where each $\Delta_k$ is iid for $k\geq 1$, and independent of $\Delta_0$.

Next, we define the increments of the process upon the observations at each $\tau_k$,
\begin{align}
\left(X_0,Y_0,\boldsymbol{\pi}_0\right)& =\left(\mathcal{N}\otimes\mathcal{W}\otimes\mathcal{P}\right)\left[0,\tau_0\right]\\
\left(X_k,Y_k,\boldsymbol{\pi}_k\right)& =\left(\mathcal{N}\otimes\mathcal{W}\otimes\mathcal{P}\right)\left(\left.\tau_{k-1},\tau_k\right]\right.
\end{align}
\noindent where $k\in\mathbb{N}$. We will also denote the values of the cumulative process upon the observations
\begin{equation}
N_k=N\left(\tau_k\right),\,W_k=W\left(\tau_k\right),\,\boldsymbol{\mathfrak{P}}_k=\mathfrak{P}\left(\tau_k\right)
\end{equation}
Therefore, we can write the increments as
\begin{align}
X_k& =N_k-N_{k-1}\\
Y_k& =W_k-W_{k-1}\\
\boldsymbol{\pi}_k& =\boldsymbol{\mathfrak{P}}_k-\boldsymbol{%
\mathfrak{P}}_{k-1}
\end{align}
\noindent for $k\in\mathbb{N}_0$, with$\,N_{-1}=W_{-1}=0,\,\boldsymbol{\pi}_{-1}=\boldsymbol{0}$.

With the delayed observation, the joint functionals of the increments depend on the amount of time since the previous observation because some nonnegative integer number of attacks will occur during each observation epoch $\left(\left.\tau_{k-1},\tau_k\right]\right.$ whereas previously we knew the increment $\left(n_k,w_k,\boldsymbol{p}_k\right)\,$corresponded to exactly 1 attack. As such, the modified functional of the increment is
\begin{equation}
\gamma_k\left(z,v,\theta,\boldsymbol{\alpha}\right)=E\left[z^{X_k}e^{-vY_k}e^{-\theta\Delta_k}e^{\boldsymbol{\alpha\cdot\pi}_k}\right],\,k\in\mathbb{N}_0
\end{equation}
\noindent where $\left|z\right|\leq 1,\,\,$Re$\left(v\right)\geq 0,\,$Re$\left(\theta\right)\geq 0,\,\boldsymbol{\alpha}\in\mathbb{C}^n$. Note that since the increments (other than the initial one) are identically distributed, we will have just two unique joint increment transforms
\begin{align}
\gamma_0\left(z,v,\theta,\boldsymbol{\alpha}\right)& =E\left[z^{X_0}e^{-vY_0}e^{-\theta\Delta_0}e^{\boldsymbol{\alpha\cdot\pi}_0}\right]\\
\gamma\left(z,v,\theta,\boldsymbol{\alpha}\right)& =E\left[z^{X_1}e^{-vY_1}e^{-\theta\Delta_1}e^{\boldsymbol{\alpha\cdot\pi}_1}\right]
\end{align}
\noindent where (2.15) is equal to $\gamma_k$ for all $k\in\mathbb{N}$.

We will be interested in the first observation epoch when the cumulative node loss component crosses a fixed threshold $M\in\mathbb{N}$, or the cumulative weight loss component crosses a threshold $V\in\mathbb{R}_{+}$, whichever comes first. Then, we define the first observed passage index
\begin{equation}
\rho=\inf\left\{n:\left(N_n,W_n\right)\notin[0,M)\times[0,V)\right\}
\end{equation}
\noindent while $\tau_\rho$ is called the first observed passage time. We refer to $N$ and $W\,$as the active components of the process (whereas $\boldsymbol{\mathfrak{P}}$ and time are passive).

Throughout the rest of this article, we consider various marginal and semi-marginal variants of the joint functional
\begin{equation}
\varPhi\left(y,z,u,v,\eta,\theta,\boldsymbol{\alpha},\boldsymbol{\beta}\right)
=E\left[y^{N_{\rho-1}} z^{N_\rho}e^{-uW_{\rho-1}-vW_\rho}e^{-\eta\tau_{\rho-1}-\theta\tau_\rho}e^{\boldsymbol{\alpha\cdot\mathfrak{P}}_{\rho-1}+\boldsymbol{\beta\cdot\mathfrak{P}}_\rho}\right]
\end{equation}
\noindent of the cumulative number of nodes lost, cumulative weight lost, time, and additional components at the observation before (i.e. the pre-observed passage time) and after (i.e. the first observed passage time) the first threshold crossing under special assumptions.

\section{APPLICATION TO POISSONIAN ATTACK PROCESS}\label{s3}

In this section we will derive analytically tractable probabilistic results for the network exit time model under delayed observation for a special case. First, we assume attacks occur according to a marked Poisson process,
\begin{equation}
\Pi=\sum_{k\geq 1}\left(n_k,w_k,\boldsymbol{p}_k\right)\varepsilon_{t_k}
\end{equation}
\noindent where $\left\{t_1,t_2,...\right\}\,$ is a Poisson point process of rate $\lambda$ on the nonnegative real line with the marks $\left(n_k,w_k,\boldsymbol{p}_k\right)$ under observation by the delayed renewal process $\left\{\tau_0,\tau_1,\tau_2,...\right\}$ as explained in Section 2.

We will derive expressions for marginal and semi-marginal versions of the functional $\varPhi$ as well as show how they lead to probabilistic results, such as moments and distributions, for the values of each component of the process in the random vicinities of the exit time.

By assuming Poisson attack occurrences, we can make the joint transforms of the increments of the process (2.14-2.15) more explicit (see Lemma A.1 from Appendix A):
\begin{align}
\gamma_0\left(z,v,\theta,\boldsymbol{\beta}\right)& =L_0\left[\theta+\lambda-\lambda g\left(zl\left(v\right)m\left(\boldsymbol{\beta}\right)\right)\right]\\
\gamma\left(z,v,\theta,\boldsymbol{\beta}\right)& =L\left[\theta+\lambda-\lambda g\left(zl\left(v\right)m\left(\boldsymbol{\beta}\right)\right)\right]
\end{align}
\noindent where $L_0$ and $L$ are the Laplace-Stieltjes transforms of $\Delta_0\,$ and $\Delta_k\,$(for $k\in\mathbb{N}$), respectively.

The following transforms will be a useful tool through which we derive analytically tractable results. Denote
\begin{equation}
D_{pq}=\mathcal{L}\mathcal{C}_p\circ\mathcal{D}_q
\end{equation}
\noindent Here $\mathcal{L}\mathcal{C}_p$ is the Laplace-Carson transform:
\begin{equation}
\mathcal{L}\mathcal{C}_p\left(\cdot\right)\left(w\right)=w\int_{p=0}^\infty e^{-wp}\left(\cdot\right)dp,\text{Re}\left(w\right)>0
\end{equation}
\noindent with the inverse
\begin{equation}
\mathcal{L}\mathcal{C}_w^{-1}\left(\cdot\right)\left(p\right)=\mathcal{L}_w^{-1}\left(\cdot\frac{1}{w}\right)\left(p\right)
\end{equation}
\noindent where $\mathcal{L}_w^{-1}$ is the inverse of the Laplace transform. The operator $\mathcal{D}_q$ is defined as
\begin{equation}
\mathcal{D}_q\left(f\right)\left(x\right)=\left(1-x\right)\sum_{q=0}^\infty x^qf\left(q\right),\left\|x\right\|<1
\end{equation}
\noindent where $\left\{f\left(q\right)\right\}$ is a sequence, with the inverse (for $r\in\mathbb{N}$)
\begin{equation}
\mathcal{D}_x^r\left(\varphi\left(x,w\right)\right)=\lim_{x\rightarrow 0}\frac{1}{r!}\frac{\partial^r}{\partial x^r}\left[\frac{1}{1-x}\varphi\left(x,w\right)\right]
\end{equation}
\noindent The inverse of D$_{pq}$ is denoted
\begin{equation}
\mathcal{D}_{xw}^{-1}\left(\cdot\right)=\mathcal{L}\mathcal{C}_w^{-1}\circ\mathcal{D}_x^{q-1}\left(\cdot\right)\left(p,q\right)
\end{equation}
\noindent According to [10], under the assumption $\left(N\left(0\right),W\left(0\right),\boldsymbol{\mathfrak{P}}\left(0\right)\right)=\left(0,0,\textbf{0}\right)$,
\begin{equation}
\begin{split}
\varPhi=\varPhi\left(1,z,0,v,0,\theta,\boldsymbol{0},\boldsymbol{\beta}\right)&=E\left[z^{N_\rho}e^{-vW_\rho}e^{-\theta\tau_\rho}e^{\boldsymbol{\beta\cdot\mathfrak{P}}_\rho}\right]\\
&=1-\left(1-\gamma\right)\mathcal{D}_{xw}^{-1}\left[\frac{1}{1-\overline{\gamma}}\right]\left(M,V\right)
\end{split}
\end{equation}
where
\begin{align}
\gamma& =\gamma\left(z,v,\theta,\boldsymbol{\beta}\right)\\
\overline{\gamma}& =\gamma\left(zx,v+w,\theta,\boldsymbol{\beta}\right)
\end{align}

\subsection{Results for a Special Case}

We will derive analytically tractable results under a special case according
to the following assumptions

\begin{enumerate}
\item Observation $\Delta_k=\tau_k-\tau_{k-1}\in\left[\text{Exponential}\left(\mu\right)\right]$, so $L\left(u\right)=\frac{\mu}{\mu+u}$

\item Nodes lost per strike $n_1\in\left[\text{Geometric}\left(a\right)\right]$ (and $b=1-a$), $g\left(z\right)=\frac{az}{1-bz}$

\item Weight lost per node lost $w_{11}\in\left[\text{Exponential}\left(\xi\right)\right]$, so $l\left(u\right)=\frac{\xi}{\xi+u}$

\item Zero initial state ($\gamma_0\equiv 1$)
\end{enumerate}

We will find explicitly under these conditions the joint functional (2.17), i.e. the joint functional of each component at the observed passage time only.

\begin{theorem}
Under assumptions 1-4,
\begin{equation}
\begin{split}
\varPhi\left(1,z,0,v,0,\theta,\boldsymbol{0},\boldsymbol{\beta}\right)& = 1-\left(1-\gamma\right)\\
& \times\left(1+\frac{b\mu}{\lambda+b\theta}+\frac{a\lambda\mu}{\left(\lambda+b\theta\right)\left(\lambda+\theta\right)}\phi\left(z,v,\theta,\boldsymbol{\beta}\right)\right)
\end{split}
\end{equation}
where
\begin{align}
\begin{split}
\phi\left(z,v,\theta,\boldsymbol{\beta}\right)& =\frac{\xi+v}{v+\xi\left(1-d\right)}-\frac{\left(d\xi\right)^M\left(1-e^{-\left(v+\xi\right)V}\sum_{j=0}^{M-2}\frac{\left[\left(\xi+v\right)V\right]^j}{j!}\right)}{\left(v+\xi\left(1-d\right)\right)\left(\xi+v\right)^{M-1}}\\
& \quad\quad -\frac{d\xi e^{-\left(\xi+v\right)V}\sum_{j=0}^{M-2}\frac{\left(d\xi V\right)^j}{j!}}{\xi-d\xi+v}
\end{split}\\
d& =zm\left(\boldsymbol{\beta}\right)\frac{\lambda+b\theta}{\lambda+\theta}
\end{align}
\end{theorem}
\begin{proof}
Using assumption 1, we have
\begin{equation*}
\gamma=L\left[\theta+\lambda-\lambda
g\left(zl\left(v\right)m\left(\boldsymbol{\beta}\right)\right)\right]=L\left[%
\theta^\ast\right]=\frac{\mu}{\mu+\theta^\ast}
\end{equation*}
where we set
\begin{align*}
\theta^\ast& =\theta+\lambda-\lambda g\left(zl\left(v\right)m\left(\boldsymbol{\beta}\right)\right)=\theta+\lambda-\lambda\frac{azl\left(v\right)m\left(\boldsymbol{\beta}\right)}{1-bzl\left(v\right)m\left(\boldsymbol{\beta}\right)}\\
\overline{\theta^\ast}& =\theta+\lambda-\lambda g\left(xzl\left(v+w\right)m\left(\boldsymbol{\beta}\right)\right)=\theta+\lambda-\lambda\frac{axzl\left(v+w\right)m\left(\boldsymbol{\beta}\right)}{1-bxzl\left(v+w\right)m\left(\boldsymbol{\beta}\right)}
\end{align*}
Combining this with (3.10),
\begin{equation}
\begin{split}
\varPhi &= 1-\left(1-\gamma\right)\mathcal{D}_{xw}^{-1}\left\{\frac{1}{1-\frac{\mu}{\mu+\overline{\theta^\ast}}}\right\}\left(M,V\right)\\
& = 1-\left(1-\gamma\right)\left(1+\mu\mathcal{D}_{xw}^{-1}\left\{\frac{1}{\overline{\theta^\ast}}\right\}\left(M,V\right)\right)
\end{split}
\end{equation}
We will explicitly calculate $\mathcal{D}_{xw}^{-1}\left\{\frac{1}{\overline{\theta^\ast}}\right\}\left(M,V\right)=\mathcal{L}\mathcal{C}_w^{-1}\circ\mathcal{D}_x^{M-1}\left\{\frac{1}{\overline{\theta^\ast}}\right\}\left(V\right)$, so first we manipulate $\frac{1}{\overline{\theta^\ast}}$ into a form for which applying $\mathcal{D}_x^{M-1}$ is possible. Firstly,
\begin{equation*}
\frac{1}{\overline{\theta^\ast}}=\frac{1}{\theta+\lambda-\lambda\frac{axzl\left(v+w\right)m\left(\boldsymbol{\beta}\right)}{1-bxzl\left(v+w\right)m\left(\boldsymbol{\beta}\right)}}=\frac{b}{\lambda+b\theta}+\frac{a\lambda}{\left(\lambda+b\theta\right)\left(\theta+\lambda\right)}\frac{1}{1-Cx}
\end{equation*}
where
\begin{equation*}
C=zl\left(v+w\right)m\left(\boldsymbol{\beta}\right)\frac{\lambda+b%
\theta}{\lambda+\theta}=c\frac{\lambda+b\theta}{\lambda+\theta}
\end{equation*}
which is constant with respect to $x$, allowing us to find
\begin{equation*}
\mathcal{D}_x^{M-1}\left\{\frac{1}{\overline{\theta^\ast}}\right\}=\frac{b}{\lambda+b\theta}+\frac{a\lambda}{\left(\lambda+b\theta\right)\left(\lambda+\theta\right)}\frac{1-C^M}{1-C}
\end{equation*}
Denote
\begin{equation}
d=\frac{C}{l\left(v+w\right)}=zm\left(\boldsymbol{\beta}\right)\frac{\lambda+b\theta}{\lambda+\theta}
\end{equation}
and calculate the inverse Laplace-Carson transform $\mathcal{L}\mathcal{C}_w^{-1}$:
\begin{align}
\begin{split}
\mathcal{L}\mathcal{C}_w^{-1}\left\{\mathcal{D}_x^{M-1}\left\{\frac{1}{\overline{\theta^\ast}}\right\}\right\}\left(V\right)& =\frac{b}{\lambda+b\theta}+\frac{a\lambda}{\left(\lambda+b\theta\right)\left(\lambda+\theta\right)}\\
& \quad \times\mathcal{L}_w^{-1}\left\{\frac{1}{w}\frac{1-d^Ml\left(v+w\right)^M}{1-dl\left(v+w\right)}\right\}\left(V\right)
\end{split}
\end{align}
After applying assumption 4 for $l$, we manipulate the expression into an appropriate form and carry out the inversion
\begin{align*}
\frac{1}{w}\frac{1-d^Ml\left(v+w\right)^M}{1-dl\left(v+w\right)}& =\frac{\xi+v}{w\left(\xi+v+w-d\xi\right)}+\frac{1}{\xi+v+w-d\xi}\\
& \quad -\frac{\left(d\xi\right)^M}{\xi-d\xi+v}\frac{1}{\left(\xi+v+w\right)^{M-1}}\left[\frac{1}{w}-\frac{1}{\xi+v+w-d\xi}\right]
\end{align*}
We are left with
\begin{align}
\begin{split}
\phi\left(z,v,\theta,\boldsymbol{\beta}\right)& =\mathcal{L}_w^{-1}\biggl\{\frac{\xi+v}{w\left(w+v+\xi\left(1-d\right)\right)}+\frac{1}{w+v+\xi\left(1-d\right)}\\
& \quad\quad\quad -\frac{\left(d\xi\right)^M}{v+\xi\left(1-d\right)}\biggl(\frac{1}{w}\frac{1}{\left(w+v+\xi\right)^{M-1}}\\
& \quad\quad\quad -\frac{1}{w+v+\xi\left(1-d\right)}\frac{1}{\left(\xi+v+w\right)^{M-1}}\biggr)\biggr\}\left(V\right)\\
& =\frac{\xi+v}{v+\xi\left(1-d\right)}-\frac{\left(d\xi\right)^M\left(1-e^{-\left(v+\xi\right)V}\sum_{j=0}^{M-2}\frac{\left[\left(\xi+v\right)V\right]^j}{j!}\right)}{\left(v+\xi\left(1-d\right)\right)\left(\xi+v\right)^{M-1}}\\
& \quad\quad -\frac{d\xi e^{-\left(\xi+v\right)V}}{v+\xi\left(1-d\right)}\sum_{j=0}^{M-2}\frac{\left(d\xi V\right)^j}{j!}
\end{split}
\end{align}
Compiling equations (3.16-3.19), we have the desired
result.
\end{proof}

\subsection{Marginal Transforms upon $\tau_\rho$}

In this section, we will find the marginal transforms of the active components and time upon the first observed passage time. First, we find the marginal PGF of the cumulative node loss component at the first observed passage time, $N_\rho$.

\begin{cor}
Under assumptions 1-4,
\begin{equation}
E\left[z^{N_\rho}\right]=\varPhi\left(1,z,0,0,0,0,\boldsymbol{0},\boldsymbol{0}\right)=\frac{a\mu\left[1-\phi^\ast\left(z,0,0,\boldsymbol{0}\right)\right]}{\lambda+\mu-\left(\lambda+b\mu\right)z}
\end{equation}
where
\begin{equation}
\phi^\ast\left(z,0,0,\boldsymbol{0}\right)=1-z^M\left(1-e^{-\xi V}\sum_{j=0}^{M-2}\frac{\left(\xi V\right)^j}{j!}\right)-e^{-\xi V}\sum_{j=0}^{M-2}\frac{\left(\xi V\right)^j}{j!}z^{j+1}
\end{equation}
\end{cor}
\noindent (i.e. $\phi^\ast\left(z,0,0,\boldsymbol{0}\right)=\left(1-z\right)\phi\left(z,0,0,\boldsymbol{0}\right)$)
\begin{proof}
We have $d=z$ and
\begin{equation*}
\theta^\ast=\lambda-\lambda
g\left(z\right)=\lambda\left(\frac{1-z}{1-bz}\right)
\end{equation*}
Therefore,
\begin{equation*}
1-\gamma=\frac{\theta^\ast}{\mu+\theta^\ast}=\lambda\left(\frac{1-z}{\mu\left(1-bz\right)+\lambda\left(1-z\right)}\right)
\end{equation*}
and
\begin{align*}
1+\frac{b\mu}{\lambda+b\theta}+\frac{a\lambda\mu}{\left(\lambda+b\theta\right)\left(\lambda+\theta\right)}\phi\left(z,v,\theta,\boldsymbol{\beta}\right)\Big|_{v=\theta=0,\boldsymbol{\beta}=\boldsymbol{0}}=\frac{\lambda+b\mu+a\mu\phi\left(z,0,0,\boldsymbol{0}\right)}{\lambda}
\end{align*}
and
\begin{align*}
\phi\left(z,0,0,\boldsymbol{0}\right)& =\frac{1}{1-z}-\frac{z^M\left(1-e^{-\xi V}\sum_{j=0}^{M-2}\frac{\left(\xi V\right)^j}{j!}\right)}{1-z}-\frac{ze^{-\xi V}\sum_{j=0}^{M-2}\frac{\left(\xi V\right)^j}{j!}z^j}{1-z}
\end{align*}
Altogether, we have
\begin{align*}
E\left[z^{N_\rho}\right]& =1-\left(\frac{\,1-z}{\mu\left(1-bz\right)+\lambda\left(1-z\right)}\right)\left(\lambda+b\mu+a\mu\phi\left(z,0,0,\boldsymbol{0}\right)\right)\\
& =\frac{a\mu\left(1-\phi^\ast\left(z,0,0,\boldsymbol{0}\right)\right)}{\lambda+\mu-\left(\lambda+b\mu\right)z}
\end{align*}
\end{proof}
Next, we find the marginal LST of the cumulative weight lost upon the first observed passage time, $W_\rho$.
\begin{cor}Under assumptions 1-4,
\begin{align}
\begin{split}
E\left[e^{-vW_\rho}\right]&=\varPhi\left(1,1,0,v,0,0,\boldsymbol{0},\boldsymbol{0}%
\right)\\
&=\frac{Ke^{-vV}}{1+kv}-\sum_{j=0}^{M-2}\frac{V^j}{j!}\frac{\xi^{M-1}e^{-(v+\xi)V}}{\left(1+kv\right)\left(v+\xi\right)^{M-1-j}}+\frac{\xi^{M-1}}{_{\left(1+kv\right)\left(v+\xi\right)^{M-1}}}
\end{split}
\end{align}
where
\begin{equation}
K=e^{-\xi V}\sum_{j=0}^{M-2}\frac{\left(\xi V\right)^j}{j!} \text{ and } k=\frac{\lambda+\mu}{a\mu\xi}
\end{equation}
\end{cor}
\begin{proof}
We have
\begin{equation*}
\theta^\ast=\lambda-\lambda g\left(l\left(v\right)\right)=\lambda\left(1-\frac{al\left(v\right)}{1-bl\left(v\right)}\right)=\frac{\lambda v}{v+a\xi}
\end{equation*}
\begin{align*}
\left(1-\gamma\right)\left(1+\frac{b\mu}{\lambda+b\theta}+\frac{a\lambda\mu}{\left(\lambda+b\theta\right)\left(\lambda+\theta\right)}\phi\left(z,v,\theta,\boldsymbol{\beta}\right)\right)\Big|_{z=1,\theta=0,\boldsymbol{\beta}=\boldsymbol{0}}\\
\quad =\frac{\lambda v+b\mu v+a\mu v\phi\left(1,v,0,\boldsymbol{0}\right)}{a\xi\mu+\left(\lambda+\mu\right)v}
\end{align*}
Since $d=zm\left(\boldsymbol{\beta}\right)\frac{\lambda+b\theta}{\lambda+\theta}=1$,
\begin{align*}
\begin{split}
\phi\left(1,v,0,\boldsymbol{0}\right)& =\frac{\xi+v}{v}-\frac{\xi^M\left(1-e^{-\left(v+\xi\right)V}\sum_{j=0}^{M-2}\frac{\left[\left(\xi+v\right)V\right]^j}{j!}\right)}{v\left(\xi+v\right)^{M-1}}-\frac{\xi e^{-\left(\xi+v\right)V}\sum_{j=0}^{M-2}\frac{\left(\xi V\right)^j}{j!}}{v}
\end{split}
\end{align*}
\begin{align*}
\begin{split}
E\left[e^{-vW_\rho}\right]& =1-\frac{\lambda v+b\mu v+a\mu v\phi\left(1,v,0,\boldsymbol{0}\right)}{a\xi\mu+\left(\lambda+\mu\right)v}=\frac{a\mu\xi+a\mu v-a\mu v\phi\left(1,v,0,\boldsymbol{0}\right)}{a\mu\xi+\left(\lambda+\mu\right)v}\\
& =\frac{1}{\xi}\frac{\xi+v-v\phi\left(1,v,0,\boldsymbol{0}\right)}{1+kv}\\
& = \frac{Ke^{-vV}}{1+kv}-\sum_{j=0}^{M-2}\frac{V^j}{j!}\frac{\xi^{M-1}e^{-(v+\xi)V}}{\left(1+kv\right)\left(v+\xi\right)^{M-1-j}}+\frac{\xi^{M-1}}{_{\left(1+kv\right)\left(v+\xi\right)^{M-1}}}
\end{split}
\end{align*}
\end{proof}
Lastly, we find the marginal Laplace-Stieltjes transform of the first observed passage time $\tau_\rho$, $\varPhi\left(1,1,0,0,0,\theta,\boldsymbol{0},\boldsymbol{0}\right)=E\left[e^{-\theta\tau_\rho}\right]$, which follows trivially from Theorem 3.1.
\begin{cor}
Under assumptions 1-4,
\begin{align}
\begin{split}
E\left[e^{-\theta\tau_\rho}\right]&=\varPhi\left(1,1,0,0,0,\theta,\boldsymbol{0},\boldsymbol{0}\right)\\
&=1-\frac{\theta}{\mu+\theta}\left[1+\frac{b\mu}{\lambda+b\theta}+\frac{a\lambda\mu}{\left(\lambda+b\theta\right)\left(\lambda+\theta\right)}\phi\left(1,0,\theta,\boldsymbol{0}\right)\right]
\end{split}
\end{align}
where
\begin{align}
\begin{split}
\phi\left(1,0,\theta,\boldsymbol{0}\right)& =\frac{1}{1-d}-\frac{d^M\left(1-e^{-\xi V\,\,}\sum_{j=0}^{M-2}\frac{\left(\xi V\right)^j}{j!}\right)}{1-d}\\
& \quad -\frac{de^{-\xi V}\sum_{j=0}^{M-2}\frac{\left(d\xi V\right)^j}{j!}}{1-d}
\end{split}\\
d=\frac{\lambda+b\theta}{\lambda+\theta}
\end{align}
\end{cor}

\subsection{Additional Probabilistic Results}

In this section, we provide a sampling of the explicit results that can be found via the marginal transforms from the previous section and demonstrate how well they match their respective random variables found via simulations of the process (as explained in Appendix B).
\begin{theorem}
Under assumptions 1-4,
\begin{equation}
E\left[N_\rho\right]=\frac{\lambda+b\mu}{a\mu}+M-\left(M-1\right)e^{-\xi V}\sum_{j=0}^{M-2}\frac{\left(\xi V\right)^j}{j!}+e^{-\xi V}\sum_{j=0}^{M-2}\frac{\left(\xi V\right)^j}{\left(j-1\right)!}
\end{equation}
\end{theorem}
\begin{proof}
Using $E[z^{N_\rho}]$ of Corollary 3.2, we can find the mean\\
$E\left[N_\rho\right]=\lim_{z\rightarrow 1-}\frac{d}{dz}E\left[z^{N_\rho}\right]$. We manipulate the PGF into a more suitable form, with $s=\frac{\lambda+\mu}{\lambda+b\mu}$,
\begin{align*}
E\left[z^{N_\rho}\right]& =\frac{a\mu}{\lambda+b\mu}\frac{1}{s-z}\left[z^M\left(1-e^{-\xi V}\sum_{j=0}^{M-2}\frac{\left(\xi V\right)^j}{j!}\right)+e^{-\xi V}\sum_{j=0}^{M-2}\frac{\left(\xi V\right)^j}{j!}z^{j+1}\right]
\end{align*}
Then,
\begin{align*}
E\left[N_\rho\right]& =\lim_{z\rightarrow 1-}\frac{d}{dz}E\left[z^{N_\rho}\right]=\left(1-K\right)\frac{M\left(s-1\right)+1}{\left(s-1\right)}+e^{-\xi V}\sum_{j=0}^{M-2}\frac{\left(\xi V\right)^j}{j!}\frac{\left[j\left(s-1\right)+s\right]}{\left(s-1\right)}\\
& =\frac{\lambda+b\mu}{a\mu}+M+\left(1-M\right)K+e^{-\xi V}\sum_{j=0}^{M-2}\frac{\left(\xi V\right)^j}{\left(j-1\right)!}
\end{align*}
\end{proof}
Next, we find the mean of the cumulative weight loss component upon the first observed passing time.
\begin{theorem}
Under assumptions 1-4,
\begin{equation}
\begin{split}
E\left[W_\rho\right]& =\frac{1}{\xi}\Bigg(\frac{\lambda+b\mu}{a\mu}+M-\left(M-1\right)e^{-\xi V}\sum_{j=0}^{M-2}\frac{\left(\xi V\right)^j}{j!}+e^{-\xi V}\sum_{j=0}^{M-2}\frac{\left(\xi V\right)^j}{\left(j-1\right)!}\Bigg)\\
& =\frac{E\left[N_\rho\right]}{\xi}=E\left[w_{11}\right]E\left[N_\rho\right]
\end{split}
\end{equation}
\end{theorem}
\begin{proof}
Using the LST $E\left[e^{-vW_\rho}\right]$ of Corollary 3.3, we use the property $E\left[W_\rho\right]=-\frac{d}{dv}E\left[e^{-vW_\rho}\right]\Big|_{v=0}$. We manipulate the LST into a suitable form
\begin{equation*}
\begin{split}
E\left[e^{-vW_\rho}\right]& =\frac{Ke^{-vV}}{1+kv}-\sum_{j=0}^{M-2}\frac{V^j}{j!}\frac{\xi^{M-1}e^{-(v+\xi)V}}{\left(1+kv\right)\left(v+\xi\right)^{M-1-j}}+\frac{\xi^{M-1}}{_{\left(1+kv\right)\left(v+\xi\right)^{M-1}}}
\end{split}
\end{equation*}
Then, applying the derivative and limit,
\begin{equation*}
\begin{split}
E\left[W_\rho\right]&=\frac{1}{\xi}\left[\frac{\lambda+b\mu}{a\mu}+M-\left(M-1\right)K+e^{-\xi V}\sum_{j=0}^{M-2}\frac{\left(\xi V\right)^j}{\left(j-1\right)!}\right]
\end{split}
\end{equation*}
\end{proof}
Simulations of the process agree with the above formulas. For each set of parameters, we generated 1,000 realizations of the process by the method described in Appendix B, yielding the following results (with sample means and absolute errors):

\begin{center}
\begin{tabular}[t]{|l|l|l|l|l|l|l|}
\hline
$\left(\lambda,\mu,a,\xi,M,V\right)$&$E\left[N_\rho\right]$&S. Mean&Error&$E\left[W_\rho\right]$&S. Mean&Error\\
\hline
(.2, 2, .5, 1, 1000, 1000)&989.08&988.82&0.26&989.08&990.06&0.98\\
\hline
(1, 2, .5, 1, 1000, 1000)&989.88&989.01&0.87&989.88&989.30&0.58\\
\hline
(3, 2, .5, 1, 1000, 1000)&991.88&990.59&1.29&991.88&989.18&2.70\\
\hline
(1, 2, .4, 1, 1000, 1000)&990.63&990.39&0.24&990.63&990.27&0.36\\
\hline
(1, 2, .2, 1, 1000, 1000)&994.38&994.04&0.34&994.38&994.86&0.48\\
\hline
(1, 2, .5, 1, 1000, 1000)&998.88&992.70&6.18&998.88&994.99&3.89\\
\hline
(1, 1, .5, 1, 1000, 1000)&990.88&990.00&0.88&990.88&989.71&1.17\\
\hline
(1, 5, .5, 1, 1000, 1000)&989.28&989.92&0.64&989.28&988.97&0.31\\
\hline
(1, 10, .5, 1, 1000, 1000)&989.08&989.08&0.00&989.08&989.68&0.31\\
\hline
(1, 2, .5, .5, 1000, 1000)&503.00&502.73&0.27&1006.00&1005.04&0.96\\
\hline
(1, 2, .5, 1.01, 1000, 1000)&994.09&993.06&1.03&984.25&983.65&0.60\\
\hline
(1, 2, .5, 2, 1000, 1000)&1002.00&1001.57&0.43&501.00&500.91&0.09\\
\hline
(1, 2, .5, 1, 1000, 800)&803.00&802.68&0.32&803.00&802.67&0.33\\
\hline
(1, 2, .75, 1, 1000, 750)&752.00&752.10&0.10&752.00&751.68&0.32\\
\hline
(1, 2, .5, .5, 500, 1000)&493.57&493.46&0.11&987.14&986.59&0.55\\
\hline
\end{tabular}
\end{center}

The highest error occurs for $E\left[N_\rho\right]$ with parameter set 6, but the relative error even here is only 0.00619, and it is typically much smaller.

Lastly, we will derive the PDF of the first observed passage time $\tau_{\rho.}$
\begin{theorem}
Under assumptions 1-4,
\begin{equation}
\begin{split}
F_{\tau_\rho}\left(\vartheta\right)& =P\left\{\tau_\rho\leq\vartheta\right\}\\
&=\lambda\left(1-K\right)\sum_{i=0}^{M-1}c_i\phi_i\left(\vartheta\right)+\lambda e^{-\xi V}\sum_{j=0}^{M-2}\frac{\left(\xi V\right)^j}{j!}\,\,\sum_{i=0}^jd_i\phi_i\left(\vartheta\right)
\end{split}
\end{equation}
where
\begin{align}
K&=e^{-\xi V}\sum_{j=0}^{M-2}\frac{\left(\xi V\right)^j}{j!}\\
c_i&=\binom{M-1}{i}\left(a\lambda\right)^ib^{M-1-i}\\
d_i&=\binom{j}{i}\left(a\lambda\right)^ib^{j-i}
\end{align}
\begin{equation}
\begin{split}
\phi_i\left(\vartheta\right)& =\frac{1}{\lambda^{i+1}}\left(1-e^{-\lambda\vartheta}\sum_{r=0}^i\frac{\left(\lambda\vartheta\right)^r}{r!}\right)\\
& \quad -\frac{e^{-\mu\vartheta}}{\left(\lambda-\mu\right)^{i+1}}\left(1-e^{-\left(\lambda-\mu\right)\vartheta}\sum_{r=0}^i\frac{\left[\left(\lambda-\mu%
\right)\vartheta\right]^r}{r!}\right)
\end{split}
\end{equation}
and $\lambda,$ $\mu,$ and $\frac{\lambda}{b}$ are distinct.
\end{theorem}
\begin{proof}
We will use the LST $E\left[e^{-\theta\tau_\rho}\right]$ from Corollary 3.4 and
\begin{equation*}
F_{\tau_\rho}\left(\vartheta\right)=\mathcal{L}_\theta^{-1}\left\{\frac{E\left[e^{-\theta\tau_\rho}\right]}{\theta}\right\}\left(\vartheta\right)
\end{equation*}
So we have
\begin{equation*}
\begin{split}
F_{\tau_\rho}\left(\vartheta\right)&=\mathcal{L}_\theta^{-1}\left\{\frac{1}{\theta}-\frac{1}{\theta}\frac{\theta}{\mu+\theta}\left(1+\frac{b\mu}{\lambda+b\theta}+\frac{a\lambda\mu\phi\left(1,0,\theta,\boldsymbol{0}\right)}{\left(\lambda+b\theta\right)\left(\lambda+\theta\right)}\right)\right\}\left(\vartheta\right)\\
&=1-\frac{\lambda e^{-\mu\vartheta}-b\mu e^{-\frac{\lambda}{b}\vartheta}}{\lambda-b\mu}-\frac{a\lambda\mu}{b}\mathcal{L}_\theta^{-1}\left\{\frac{\phi\left(1,0,\theta,\boldsymbol{0}\right)}{\left(\theta+\mu\right)\left(\theta+\lambda\right)\left(\theta+\frac{\lambda}{b}\right)}\right\}\left(\vartheta\right)
\end{split}
\end{equation*}
and $1-d=\frac{\lambda+\theta-\lambda-b\theta}{\lambda+\theta}=\frac{a\theta}{\lambda+\theta}$. Using this and the formula above, we can write $\phi\left(1,0,\theta,\boldsymbol{0}\right)$ in a more convenient form
\begin{equation*}
\phi\left(1,0,\theta,\boldsymbol{0}\right)=\frac{\theta+\lambda}{a\theta}-\frac{1-K}{a}\frac{\left(b\theta+\lambda\right)^M}{\theta\left(\theta+\lambda\right)^{M-1}}-\frac{e^{-\xi V}}{a}\sum_{j=0}^{M-2}\frac{\left(\xi V\right)^j}{j!}\frac{\left(b\theta+\lambda\right)^{j+1}}{\theta\left(\theta+\lambda\right)^j}
\end{equation*}
Thus, the remaining inverse transform can be separated into three parts
\begin{align}
&\frac{1}{a}\mathcal{L}_\theta^{-1}\left\{\frac{1}{\theta\left(\theta+\mu\right)\left(\theta+\frac{\lambda}{b}\right)}\right\}\left(\vartheta\right)\\
&-b\frac{1-K}{a}\mathcal{L}_\theta^{-1}\left\{\frac{\left(b\theta+\lambda\right)^{M-1}}{\theta\left(\theta+\mu\right)\left(\theta+\lambda\right)^M}\right\}\left(\vartheta\right)\\
&-\frac{be^{-\xi V}}{a}\sum_{j=0}^{M-2}\frac{\left(\xi V\right)^j}{j!}\mathcal{L}_\theta^{-1}\left\{\frac{\left(b\theta+\lambda\right)^j}{\theta\left(\theta+\mu\right)\left(\theta+\lambda\right)^{j+1}}\right\}\left(\vartheta\right)
\end{align}
Suppose $\mu,\lambda,$ and $\frac{\lambda}{b}$ are distinct, then we can give a more explicit form of the inversion. With this, we can do the inversion in (3.34):

\begin{equation*}
\mathcal{L}_\theta^{-1}\left\{\frac{1}{\theta\left(\theta+\mu\right)\left(\theta+\frac{\lambda}{b}\right)}\right\}\left(\vartheta\right)=\frac{b}{\mu\lambda}\left(1-\frac{\lambda e^{-\mu\vartheta}-b\mu e^{-\frac{\lambda}{b}\vartheta}}{\lambda-b\mu}\right)
\end{equation*}

Next, the inversion from (3.35)
$\mathcal{L}_\theta^{-1}\left\{\frac{\left(b\theta+\lambda\right)^{M-1}}{\theta\left(\theta+\mu\right)\left(\theta+\lambda\right)^M}\right\}\left(\vartheta\right)=e^{-\lambda\vartheta}\mathcal{L}_\theta^{-1}\left\{\frac{\left(b\theta+a\lambda\right)^{M-1}}{\left(\theta-\lambda\right)\left(\theta+\mu-\lambda\right)\theta^M}\right\}\left(\vartheta\right)$
\begin{align*}
& =\frac{e^{-\lambda\vartheta}}{\mu}\mathcal{L}_\theta^{-1}\left\{%
\sum_{i=0}^{M-1}c_i\left[\frac{1}{\theta^{i+1}\left(\theta+\mu-\lambda%
\right)}-\frac{1}{\theta^{i+1}\left(\theta-\lambda\right)}\right]\right\}%
\left(\vartheta\right)\\
\begin{split}
& =\frac{1}{\mu}\sum_{i=0}^{M-1}c_i\Biggl[\frac{1}{\lambda^{i+1}}\left(1-e^{-\lambda\vartheta}\sum_{r=0}^i\frac{\left(\lambda\vartheta\right)^r}{r!}\right)\\
&\quad\quad\quad\quad\quad\quad -\frac{e^{-\mu\vartheta}}{\left(\lambda-\mu\right)^{i+1}}\left(1-e^{-\left(\lambda-\mu\right)\vartheta}\sum_{r=0}^i\frac{\left[\left(\lambda-\mu\right)\vartheta\right]^r}{r!}\right)\Biggr]
\end{split}
\end{align*}
The inversion of (3.36) is the same as the previous with $M=j+1$. Combining the completed inverse transform with these yields
\begin{equation*}
F_{\tau_\rho}\left(\vartheta\right)=\lambda\left(1-K\right)\sum_{i=0}^{M-1}c_i\phi_i\left(\vartheta\right)+\lambda e^{-\xi V}\sum_{j=0}^{M-2}\frac{\left(\xi V\right)^j}{j!}\sum_{i=0}^jd_i\phi_i\left(\vartheta\right)
\end{equation*}
\end{proof}


\bibliographystyle{plain}
\bibliography{ORSN}
\nocite{Agarwal:2005uq,Agarwal:2004uq,Agarwal:2004fk,arbor-report,Bingham:2001kx,Bollobas:2001vn,Callaway:2000ys,Costa:2008zr,Crescenzo:2009ly,Dshalalow:1997ve,Dshalalow:2001qf,Dshalalow:2005bh,Dshalalow:2006oq,Dshalalow:2008dq,Dshalalow:2009cr,Dshalalow:2044nx,Erdos:1959kl,Erdos:1960tg,Kadankov:2005hc,Kadankova:2007ij,Kalisky:2006bs,Kyprianou:2003fv,Mellander:1992dz,Muzy:2007fu,Newman:2001kl,Newman:2002qa,Newman:2003mi,Newman:2004ff,Newman:2004pi,Newman:2006lh,Takacs:1978fu,Telcs:1989ye,Wu:2982qo}

\appendix
\begin{appendices}
\section{Proof of the Value of $\gamma_k\left(z,v,\theta,\boldsymbol{\beta}\right)$}

\begin{lemma}
For $k\in\mathbb{N}$,
\begin{equation}
\gamma_k\left(z,v,\theta,\boldsymbol{\beta}\right)=L\left[\theta+\lambda-\lambda g\left(zl\left(v\right)m\left(\boldsymbol{\beta}\right)\right)\right]
\end{equation}
\end{lemma}
\begin{proof}
Since the increments $\left(X_k,Y_k,\boldsymbol{\pi}_k\right)$ are iid for $k\in\mathbb{N}$ and stationary, and the observation intervals $\Delta_k=\tau_k-\tau_{k-1}$ are iid for $k\in\mathbb{N}$, we have a common joint functional for all $k\in\mathbb{N}$, we have
\begin{equation*}
\begin{split}
\gamma\left(z,v,\theta,\boldsymbol{\beta}\right)& =E\left[z^{X_1}e^{-vY_1}e^{-\theta\Delta_1}e^{\boldsymbol{\beta\cdot\pi}_1}\right]=E\left[e^{-\theta\Delta_1}E\left[z^{X_1}e^{-vY_1}e^{\boldsymbol{\beta\cdot\pi}_1}\Big|\Delta_1\right]\right]
\end{split}
\end{equation*}
Let $J=\Lambda\left(\Delta_1\right)$ be the number of strikes during an observation epoch $(\tau_0,\tau_1]$. Then, 
\begin{equation*}
\begin{split}
& =E\left[e^{-\theta\Delta_1}E\left[\left(z^{n_1}e^{-vw_1}e^{\boldsymbol{\beta\cdot p}_1}\right)\times\cdots\times\left(z^{n_J}e^{-vw_J}e^{\boldsymbol{\beta\cdot p}_J}\right)\Big|\Delta_1\right]\right]\\
& =E\biggl[e^{-\theta\Delta_1}E\biggl[z^{n_1}e^{-v\left(w_{11}+...+w_{n_11}\right)}e^{\boldsymbol{\beta\cdot}\left(\boldsymbol{p}_{11}+...\boldsymbol{p}_{n_11}\right)}\times\cdots\\
& \quad\quad\quad\quad\quad\quad\quad\times z^{n_J}e^{-v\left(w_{1J}+...+w_{n_JJ}\right)}e^{\boldsymbol{\beta\cdot}\left(\boldsymbol{p}_{1J}+...\boldsymbol{p}_{n_JJ}\right)}\bigg|\Delta_1\biggr]\biggr]
\end{split}
\end{equation*}
Since $w_{jk}$'s and $\boldsymbol{p}_{jk}$'s are iid for $k\geq 1$,
\begin{equation*}
=E\left[e^{-\theta\Delta_1}E\left[g\left(zl\left(v\right)m\left(\boldsymbol{\beta}\right)\right)^J\Big|\Delta_1\right]\right]
\end{equation*}
Since $\Lambda$ is a Poisson counting measure, $J=\Lambda\left(\Delta_1\right)\in\left[\text{Poisson(}\lambda\Delta_1\text{)}\right]$, so we know$\,E\left[z^J\Big|\Delta_1\right]=e^{\lambda\Delta_1\left(z-1\right)}$ that yields
\begin{equation*}
\begin{split}
\gamma\left(z,v,\theta,\boldsymbol{\beta}\right)& =E\left[e^{-\theta\Delta_1}e^{\lambda\Delta_1\left[g\left(zl\left(v\right)m\left(\boldsymbol{\beta}\right)\right)-1\right]}\right]=E\left[e^{-\left[\theta+\lambda-\lambda g\left(zl\left(v\right)m\left(\boldsymbol{\beta}\right)\right)\right]\Delta_1}\right]\\
&=L\left[\theta+\lambda-\lambda g\left(zl\left(v\right)m\left(\boldsymbol{\beta}\right)\right)\right]
\end{split}
\end{equation*}
where $L$ is the LST of $\Delta_1$
\end{proof}

Similarly, $\gamma_0\left(z,v,\theta,\boldsymbol{\beta}\right)=L_0\left[\theta+\lambda-\lambda g\left(zl\left(v\right)m\left(\boldsymbol{\beta}\right)\right)\right]$ where $L_0$ is the LST of $\Delta_0$.

\section{Simulation}

The following is a high-level overview of one simulation of the process until the first observed passage time for a particular set of parameters $\left(\lambda,\mu,a,\xi,M,V\right)$:

\begin{algorithmic}
\State cumulativeNodeLoss $\gets 0$
\State cumulativeWeightLoss $\gets 0$
\State cumulativeTimePassed $\gets 0$

\While{cumulativeNodeLoss $<M$ and cumulativeWeightLoss $<V$}\
	\State observationTime $\gets$ Exponential($\mu$) R.V.
	\State strikesInEpoch $\gets$ Poisson($\lambda *$observationTime) R.V.\\
	\State nodesLostInEpoch $\gets 0$
	\State weightLostInEpoch $\gets 0$\\
	
	\State nodesLostInEpoch $\gets \sum_{i=1}^{\text{StrikesInEpoch}}X_i$ ($X_i\in\left[\text{Geometric}\left(a\right)\right]$)
	\State weightLostInEpoch $\gets \sum_{i=1}^{%
\text{StrikesInEpoch}}\sum_{j=1}^{X_i}Y_{ij}$ ($Y_{ij}\in\left[\text{Exponential}\left(\xi\right)\right]$)\\

	\State cumulativeNodeLoss $\gets$ cumulativeNodeLoss + nodesLostInEpoch
	\State cumulativeWeightLoss $\gets$ cumulativeWeightLoss + weightLostInEpoch
	\State cumulativeTimePassed $\gets$ cumulativeTimePassed + observationTime\\
	
	\If {cumulativeNodeLoss $\geq M$ or cumulativeWeightLoss $\geq W$}
		\State CrossingValues[1]=cumulativeTimePassed
		\State CrossingValues[2]=cumulativeNodeLoss
		\State CrossingValues[3]=cumulativeWeightLoss
	\EndIf
\EndWhile
\end{algorithmic}

In other words, we generate an observation time, generate the number of attacks within the observation time, generate the number of nodes lost in each attack, generate a weight for each node lost, and repeat with successive observation epochs until the first threshold is crossed, at which time we record the crossing values of each component.

While this code generates the observation before the attacks rather than generating attacks and then observing them (which we are actually modeling), the independent increments property of the attack process (Poisson point process) renders this strategy probabilistically equivalent and yields simpler code.

In the numbers provided in Section 3, we generate a sample by running the simulation many times and average the crossing values, each of which are iid random variables with finite mean, so each converges almost surely to the true mean by the strong law of large numbers.
\section{Validation that $F_{\tau_\rho}$ is a PDF}

We will next prove that $F_{\tau_\rho}$ of Theorem 3.7 is a PDF as some supporting confirmation of the many calculations leading up to it.

Since $\phi_i\left(0\right)=0$, $F_{\tau_\rho}\left(0\right)=0$. Since $\lim_{\vartheta\rightarrow\infty}e^{-\lambda\vartheta}\vartheta^r=0\,$for all $r\in\mathbb{N}_0$, we have
\begin{equation}
\lim_{\vartheta\rightarrow\infty}\phi_i\left(\vartheta\right)=\frac{1}{\lambda^{i+1}}\left(1-0-0\right)-\frac{1}{\left(\lambda-\mu\right)^{i+1}}\left(0-0-0\right)=\frac{1}{\lambda^{i+1}}
\end{equation}
Since the rest of $F_{\tau_\rho}\left(\vartheta\right)$ is independent of $\vartheta$, we find
\begin{equation}
\begin{split}
\lim_{\vartheta\rightarrow\infty}\sum_{i=0}^{M-1}c_i\phi_i\left(\vartheta\right)&=\sum_{i=0}^{M-1}\binom{M-1}{i}\frac{\left(a\lambda\right)^ib^{M-1-i}}{\lambda^{i+1}}\\
& =\frac{1}{\lambda}\sum_{i=0}^{M-1}\binom{M-1}{i}a^ib^{M-1-i}=\frac{\left(a+b\right)^{M-1}}{\lambda}=\frac{1}{\lambda}
\end{split}
\end{equation}
and the same follows for $\sum_{i=0}^jd_i\phi_i\left(\vartheta\right)$, so we have
\begin{equation}
\lim_{\vartheta\rightarrow\infty}F_{\tau_\rho}\left(\vartheta\right)=\left(1-K\right)+K=1
\end{equation}
\indent Lastly, we show that $F_{\tau_\rho}\left(\vartheta\right)$ is monotone increasing. First, note that $K$ is the probability that a Poisson$\left(\xi V\right)$ R.V. is less than or equal to $M-2$, so $0<1-K<1$. Thus, if $\phi_i\left(\vartheta\right)$ is monotone increasing (which would imply $\phi_i\left(\vartheta\right)$ is nonnegative since $\phi_i\left(0\right)=0$) for each $i$, then $F_{\tau_\rho}\left(\vartheta\right)$ is monotone increasing.

\begin{lemma}
$\phi_i\left(\vartheta\right)$ is monotone increasing.
\end{lemma}
\begin{proof}
First, we find the derivative with respect to $\vartheta$
\begin{equation}
\begin{split}
\phi_i'\left(\vartheta\right)& =\frac{1}{\lambda^{i+1}}\left(\lambda e^{-\lambda\vartheta}\sum_{r=0}^i\frac{\left(\lambda\vartheta\right)^r}{r!}-e^{-\lambda\vartheta}\sum_{r=1}^i\frac{\lambda^r}{r!}r\vartheta^{r-1}\right)\\
& \quad +\frac{\mu e^{-\mu\vartheta}}{\left(\lambda-\mu\right)^{i+1}}\left(1-e^{-\left(\lambda-\mu\right)\vartheta}\sum_{r=0}^i\frac{\left[\left(\lambda-\mu\right)\vartheta\right]^r}{r!}\right)\\
& \quad -\frac{e^{-\mu\vartheta}}{\left(\lambda-\mu\right)^{i+1}}\biggl(\left(\lambda-\mu\right)e^{-\left(\lambda-\mu\right)\vartheta}\sum_{r=0}^i\frac{\left[\left(\lambda-\mu\right)\vartheta\right]^r}{r!}\\
& \quad\quad\quad\quad\quad\quad\quad\quad -e^{-\left(\lambda-\mu\right)\vartheta}\sum_{r=1}^i\frac{\left(\lambda-\mu\right)^r}{r!}r\vartheta^{r-1}\biggr)
\end{split}
\end{equation}
Clearly $\lim_{\vartheta\rightarrow\infty}\phi_i'\left(\vartheta\right)=0\,$ as the terms of the type $e^{-c\vartheta}\,$(with $c\geq 0$) dominate the polynomial terms. We will consider the first set of parentheses first
\begin{align*}
&\frac{1}{\lambda^{i+1}}\left(\lambda e^{-\lambda\vartheta}\sum_{r=0}^i\frac{\left(\lambda\vartheta\right)^r}{r!}-e^{-\lambda\vartheta}\sum_{r=1}^i\frac{\lambda^r}{r!}r\vartheta^{r-1}\right)\\
&\quad =\frac{\lambda e^{-\lambda\vartheta}}{\lambda^{i+1}}\left(\sum_{r=0}^i\frac{\left(\lambda\vartheta\right)^r}{r!}-\lambda\sum_{r=0}^{i-1}\frac{\left(\lambda\vartheta\right)^r}{r!}\right)=\frac{e^{-\lambda\vartheta}}{\lambda^i}\frac{\left(\lambda\vartheta\right)^i}{i!}=\frac{e^{-\lambda\vartheta}\vartheta^i}{i!}
\end{align*}
The third set of parenthesis can be evaluated similarly
\begin{align*}
&-\frac{e^{-\mu\vartheta}}{\left(\lambda-\mu\right)^{i+1}}\Biggl(\left(\lambda-\mu\right)e^{-\left(\lambda-\mu\right)\vartheta}\sum_{r=0}^i\frac{\left[\left(\lambda-\mu\right)\vartheta\right]^r}{r!}-e^{-\left(\lambda-\mu\right)\vartheta}\sum_{r=1}^i\frac{\left(\lambda-\mu\right)^r}{r!}r\vartheta^{r-1}\Biggr)\\
&\quad\quad =-\frac{e^{-\mu\vartheta}e^{-\left(\lambda-\mu\right)\vartheta}%
\vartheta^i}{i!}=-\frac{e^{-\lambda\vartheta}\vartheta^i}{i!}
\end{align*}
These sum to zero, leaving only the second set of parentheses to consider. If we can show the remaining term is nonnegative for finite $\vartheta$, the proof will be complete:
\begin{equation}
\frac{\mu e^{-\mu\vartheta}}{\left(\lambda-\mu\right)^{i+1}}\left(1-e^{-\left(\lambda-\mu\right)\vartheta}\sum_{r=0}^i\frac{\left[\left(\lambda-\mu\right)\vartheta\right]^r}{r!}\right)
\end{equation}
\noindent\underline{Case 1 $\lambda>\mu$}: In this case, $\frac{\mu
e^{-\mu\vartheta}}{\left(\lambda-\mu\right)^{i+1}}>0$. We also have
\begin{equation*}
1\leq\sum_{r=0}^i\frac{\left[\left(\lambda-\mu\right)\vartheta\right]^r}{r!}<e^{\left(\lambda-\mu\right)\vartheta}
\end{equation*}
as above since $\lambda-\mu>0$, and so
\begin{align*}
& 0<e^{-\left(\lambda-\mu\right)\vartheta}\leq e^{-\left(\lambda-\mu\right)\vartheta}\sum_{r=0}^i\frac{\left[\left(\lambda-\mu\right)\vartheta\right]^r}{r!}<1\\
& 1-e^{-\left(\lambda-\mu\right)\vartheta}\sum_{r=0}^i\frac{\left[\left(\lambda-\mu\right)\vartheta\right]^r}{r!}>0
\end{align*}
Thus, the whole term is nonnegative.\vspace{5mm}

\noindent\underline{Case 2 $\lambda<\mu$}: By Taylor's theorem,
\begin{equation*}
e^{\left(\lambda-\mu\right)\vartheta}=\sum_{r=0}^i\frac{\left[\left(\lambda-\mu\right)\vartheta\right]^r}{r!}+\frac{\left(-1\right)^{i+1}\left[\left(\mu-\lambda\right)\vartheta\right]^{i+1}e^{\left(\lambda-\mu\right)\omega}}{\left(i+1\right)!}
\end{equation*}
for some $0<\omega<\vartheta$.\vspace{5mm}

\noindent\underline{Case 2.1 $i$ is even}: The error term is negative in this case, so the partial sum is more than $e^{\left(\lambda-\mu\right)\vartheta}$, implying
\begin{align*}
&e^{-\left(\lambda-\mu\right)\vartheta}\sum_{r=0}^i\frac{\left[\left(\lambda-\mu\right)\vartheta\right]^r}{r!}>1\\
& 1-e^{-\left(\lambda-\mu\right)\vartheta}\sum_{r=0}^i\frac{\left[\left(\lambda-\mu\right)\vartheta\right]^r}{r!}<0
\end{align*}

and $\frac{\mu e^{-\mu\vartheta}}{\left(\lambda-\mu\right)^{i+1}}<0$, thus implying the whole term is positive.\vspace{5mm}

\noindent\underline{Case 2.2 $i$ is odd}: The error term is positive in this case, so the partial sum is less than $e^{\left(\lambda-\mu\right)\vartheta}$, which implies

\begin{align*}
& 0<e^{-\left(\lambda-\mu\right)\vartheta}\sum_{r=0}^i\frac{\left[\left(\lambda-\mu\right)\vartheta\right]^r}{r!}<1\\
& 1-e^{-\left(\lambda-\mu\right)\vartheta}\sum_{r=0}^i\frac{\left[\left(\lambda-\mu\right)\vartheta\right]^r}{r!}>0
\end{align*}
and $\frac{\mu e^{-\mu\vartheta}}{\left(\lambda-\mu\right)^{i+1}}>0$, thus implying the whole term is positive. 
\end{proof}
Altogether, we have verified $F_{\tau_\rho}\left(\vartheta\right)$ is a probability distribution function.
\end{appendices}

\end{document}